\documentclass[11pt]{amsart}
\usepackage{graphicx}
\usepackage{amssymb, amsmath}
\newtheorem{theorem}{Theorem}[section]

\theoremstyle{definition}

\theoremstyle{remark}

\numberwithin{equation}{section}

\newcommand{\Gf}{\mathrm{der}_{\theta, b}(G, \cdot)}

\begin{document}
\title{On free polyadic groups}
\author{Gh. Fathtabar}
\address{Gh. Fathtabar} \email{fathtabar@kashanu.ac.ir}
\author{\sc H. Khodabandeh}
\address{H. Khodabandeh} \email{hamid$_{-}$2794@yahoo.com}
\author{\sc K.Yousefi}
\address{K.Yousefi} \email{kosar$_{-}$u1368@yahoo.com}
\thanks{{\scriptsize
\hskip -0.4 true cm MSC(2010): Primary 20N15, Secondary 08A99 and 14A99
}}

\address{K. Yousefi} 
\date{\today}

\begin{abstract}
In this article, for a polyadic group $(G, f)=\Gf$, we give a necessary and sufficient condition in terms of the group $(G, \cdot)$, the automorphism $\theta$, and the element $b$, in order that the polyadic group becomes free.
\end{abstract}

\maketitle


\section{Introduction}

In this article, we continue our study of the structure of free {\em polyadic groups} and we will answer a question of M. Shahryari. A polyadic group is a  natural generalization of the concept of group to the case where the binary operation of group replaced with an $n$-ary associative operation, one variable linear equations in which have unique solutions (see the next section for the detailed definitions). These interesting algebraic objects are introduced by Kasner and D\"ornte (\cite{Kas} and \cite{Dor}) and  studied extensively by Emil Post during the first decades of the last century, \cite{Post}. During decades, many articles are published on the structure of polyadic groups. Already  homomorphisms and automorphisms of polyadic groups are studied in  \cite{Khod-Shah}. A characterization of the simple polyadic groups is obtained by them in \cite{Khod-Shah2}. Also  representation theory of polyadic groups is studied in  \cite{Dud-Shah}. The complex characters of finite polyadic groups are also investigated by the M. Shahryari in \cite{Shah2}.

It is known that for every polyadic ($n$-ary) group $(G, f)$, there exists a corresponding ordinary group $(G, \cdot)$, an automorphism $\theta$ of this ordinary group, and an element $b\in G$, the structure of $(G, f)$ in which complectly determined by  $(G, \cdot)$, $\theta$, and $b$. M. Shahryari, asked us to find a necessary and sufficient condition for a polyadic group $(G, f)$ to be free in terms of  the group $(G, \cdot)$, the automorphism $\theta$, and the element $b$. Our aim is to give an answer to this question.

\section{Polyadic groups}
This section contains basic notions and properties of polyadic groups as well as some literature.
Let $G$ be a non-empty set and $n$ be a positive integer. If
$f:G^n\to G$ is an $n$-ary operation, then we use the compact
notation $f(x_1^n)$ for the elements $f(x_1, \ldots, x_n)$. In
general, if $x_i, x_{i+1}, \ldots, x_j$ is an arbitrary sequence of
elements in $G$, then we denote it as $x_i^j$. In the special case,
when all terms of this sequence are equal to a constant $x$, we denote
it by $\stackrel{(t)}{x}$, where $t$ is the number of terms.  We say that an $n$-ary
operation is {\em associative}, if for any $1\leq i<j\leq n$, the
equality
$$
f(x_1^{i-1},f(x_i^{n+i-1}),x_{n+i}^{2n-1})=
f(x_1^{j-1},f(x_j^{n+j-1}),x_{n+j}^{2n-1})
$$
holds for all $x_1,\ldots,x_{2n-1}\in G$. An $n$-ary system $(G,f)$
is called an {\em $n$-ary group} or a {\em polyadic group}, if $f$
is associative and for all $a_1,\ldots,a_n, b\in G$ and $1\leq i\leq
n$, there exists a unique element $x\in G$ such that
$$
f(a_1^{i-1},x,a_{i+1}^n)=b.
$$
It is proved that the uniqueness assumption on the solution $x$ can
be dropped \cite{Dud2}. Clearly, the case $n=2$ is
just the definition of ordinary groups. During
this article, we assume that $n> 2$.
The classical paper of E. Post \cite{Post},  is
one of the first articles published on the subject. In this paper,
Post proves his well-known {\em coset theorem}. Many basic
properties of polyadic groups are studied in this paper. The
articles \cite{Dor} and \cite{Kas} are among the first materials
written on the polyadic groups. Russian reader, can use
the book of Galmak \cite{Gal}, for an almost complete description of
polyadic groups.
The articles \cite{Art}, \cite{Dud1}, \cite{Gal2}, and \cite{Gle}
can be used for study of axioms of polyadic groups as well as their
varieties.

Note that an $n$-ary system $(G,f)$ of the form $\,f(x_1^n)=x_1
x_2\ldots x_nb$, where $(G,\cdot)$ is a group and $b$  a fixed
element belonging to the center of $(G,\cdot)$, is an $n$-ary group.
Such an $n$-ary group is called {\em $b$-derived} from the group
$(G,\cdot)$ and it is denoted by $\mathrm{der}_b^n(G, \cdot)$. In the case when $b$ is the identity of $(G,\cdot)$, we
say that such a polyadic group is {\em reduced} to the group $(G, \cdot
)$ or {\em derived} from $(G, \cdot)$ and we use the notation $\mathrm{der}^n(G, \cdot)$ for it. For every $n>2$, there
are $n$-ary groups which are not derived from any group. An $n$-ary
group $(G,f)$ is derived from some group if and only if it contains
an element $a$ (called an {\em $n$-ary identity}) such that
$$
 f(\stackrel{(i-1)}{a},x,\stackrel{(n-i)}{a})=x
$$
holds for all $x\in G$ and for all $i=1,\ldots,n$, see \cite{Dud2}.

From the definition of an $n$-ary group $(G,f)$, we can directly see
that for every $x\in G$, there exists only one $y\in G$,  satisfying
the equation
$$
f(\stackrel{(n-1)}{x},y)=x.
$$
This element is called {\em skew} to $x$ and it is denoted by
$\overline{x}$.    As D\"ornte \cite{Dor} proved, the
following identities hold for all $\,x,y\in G$, $2\leq i\leq n$,
$$
f(\stackrel{(i-2)}{x},\overline{x},\stackrel{(n-i)}{x},y)=
f(y,\stackrel{(n-i)}{x},\overline{x},\stackrel{(i-2)}{x})=y.
$$
These identities together with the associativity identities, axiomatize the variety of polyadic groups in the algebraic language $(f, ^{-})$.

Suppose $(G, f)$ is a polyadic group and $a\in G$ is a fixed
element. Define a binary operation
$$
x\ast y=f(x,\stackrel{(n-2)}{a},y).
$$
Then $(G, \ast)$ is an ordinary group, called the {\em retract} of
$(G, f)$ over $a$. Such a retract will be denoted by $ret_a(G,f)$. All
retracts of a polyadic group are isomorphic \cite{DM}. The
identity of the group $(G,\ast)$ is $\overline{a}$. One can verify
that the inverse element to $x$ has the form
$$
y=f(\overline{a},\stackrel{(n-3)}{x},\overline{x},\overline{a}).
$$

One of the most fundamental theorems of polyadic group is the
following, now known as {\em Hossz\'{u} -Gloskin's Theorem}. We will use
it frequently in this article and the reader can use \cite{DG},
\cite{DM1}, \cite{Hos} and \cite{Sok} for detailed discussions.

\begin{theorem}
Let $(G,f)$ be an $n$-ary group. Then there exists an ordinary group $(G, \cdot)$,
an automorphism $\theta$ of $(G, \cdot)$ and an element $b\in G$ such that

1.  $\theta(b)=b$,

2.  $\theta^{n-1}(x)=b x b^{-1}$, for every $x\in
G$,

3.  $f(x_1^n)=x_1\theta(x_2)\theta^2(x_3)\cdots\theta^{n-1}(x_n)b$, for
all $x_1,\ldots,x_n\in G$.

\end{theorem}

According to this theorem, we  use the notation $\Gf$ for $(G,f)$
and we say that $(G,f)$ is $(\theta, b)$-derived from the group $(G,
\cdot)$. During this paper, we will assume that $(G, f)=\Gf$.

Varieties of polyadic groups and the structure of congruences on
polyadic groups are studied in \cite{Art} and \cite{Dud1}. It is proved that all congruences on polyadic groups
are commute and so the lattice of congruences is modular.

There is one more important object associated to polyadic groups. Let $(G, f)$ be a polyadic group. Then, as Post proved, there exists a unique group $(G^{\ast}, \circ)$ (which we call now the Post's cover of $(G, f)$) such that \\

{\em 1- $G$ is contained in $G^{\ast}$ as a coset of a normal subgroup $R$.\\

2- $R$ is isomorphic to a retract of $(G, f)$.\\

3- We have $G^{\ast}/R\cong \mathbb{Z}_{n-1}$.\\

4- Inside $G^{\ast}$, for all $x_1, \ldots, x_n\in G$, we have
$f(x_1^n)=x_1\circ x_2\circ \cdots \circ x_n$. \\

5- $G^{\ast}$ is generated by $G$.}\\

The group $G^{\ast}$ is also universal in the class of all groups having properties 1, 4. More precisely, if $\beta: (G, f)\to \mathrm{der}^n(H, \ast)$ is a polyadic homomorphism, then there exists a unique ordinary homomorphism $h: G^{\ast}\to H$ such that $h_{|_{G}}=\beta$. This universal property characterizes $G^{\ast}$ uniquely. The explicit construction of the Post's cover can be find in \cite{Shah2}.

\section{Free polyadic groups}
The structure of free polyadic groups, their construction, and their Post's cover are studied in \cite{Khod-Shah3}.  Let $(G,f)=\Gf$ be a polyadic group. It is natural to ask when this polyadic group is free. Our main theorem gives a complete solution for this problem. We will need the following well-known theorem of Nielsen-Schrier on the bases of free groups. Note that if $F(X)$ is the free group on a set $X$, $H$ is any subgroup, and $\mathcal{U}$ is a right transversal for $H$, then by the notation $\hat{u}$ we denote the unique element of $\mathcal{U}$ satisfying $Hu=H\hat{u}$. The proof of the following theorem can be find in any standard book of presentation theory of groups.

\begin{theorem}
Let $F(X)$ be the free group of rank $r$, $H$ be a subgroup of index $m$ and $\mathcal{U}$ be a right transversal of $H$. Then the rank of $H$ is equal to $m(r-1)+1$ and the set
$$
\{ u\widehat{(xux^{-1})}:\ u\in \mathcal{U}, x\in X, u\widehat{(xux^{-1})}\neq 1\}
$$
is a basis of $H$.
\end{theorem}

We can now prove our main theorem.

\begin{theorem}
Let $(G,f)=\Gf$ be a polyadic group. The necessary and sufficient condition for $(G,f)$ to be free of rank $s>1$ is that \\

1- the group $(G, \cdot)$ is a free  of rank $k$ and
$$
s=\frac{k-1}{n-1}+1,
$$

2- there exists a subset $\{ v_1, \ldots, v_{s-1}\}\subseteq G$, such that the set
$$
\{ b, \theta^j(v_i):\ 1\leq i\leq s-1, 0\leq j\leq n-2\}
$$
is a basis of $(G, \cdot)$.
\end{theorem}

\begin{proof}
First, assume that $(G,f)$ is free of rank $s>1$ and let
$$
X=\{ u, v_1, \ldots, v_{s-1}\}
$$
be a basis. By \cite{Khod-Shah3}, we have
$$
G=F^n_p(X)=\{ w\in F(X):\ \mathrm{ht}(w)\equiv 1\ (\mathrm{mod}\ n-1)\},
$$
where $\mathrm{ht}:F(X)\to \mathbb{Z}$ is the homomorphism $\mathrm{ht}(x_1^{\varepsilon_1}\ldots x_m^{\varepsilon_m})=\sum \varepsilon_i$.

Suppose that $H=\ker(\mathrm{ht})$. Then we have
$$
F(X)=H\cup Hu\cup \cdots \cup Hu^{n-2}.
$$
By \cite{Artam}, we know that $F(X)$ is the Post's cover of $(G,f)$ and since the index of $H$ in $F(X)$ is $n-1$, so $H$ is free of rank $k=(s-1)(n-1)+1$ and by the coset theorem of Post we have $(G, \cdot)\cong H$. So, $(G,f)$ is also a free group of rank $k=(s-1)(n-1)+1$ and this proves 1. Now, consider the right transversal $\mathcal{U}=\{ 1, u, u^2, \ldots, u^{n-2}\}$. By the above theorem, the set
$$
B=\{ u^i\widehat{(xu^ix^{-1})}\neq 1:\ 0\leq i\leq n-2, x\in X\}
$$
is a basis of $H$. Note that if $x=u$, then 
$$
u^ix\widehat{u^ix^{-1}}=u^{i+1}\widehat{u^{i+1}}^{-1}, 
$$
and hence for $0\leq i\leq n-3$, we have $u^ix\widehat{u^ix^{-1}}=1$ and for $i-n-2$, we have $u^ix\widehat{u^ix^{-1}}=u^{n-1}$. If $x=v_j$, then 
$$
u^iv_j\widehat{u^iv_j^{-1}}=u^iv_ju^{-(i+1)}, 
$$
and since $\mathrm{ht}(u^iv_j)=j+1$, so we have $u^iv_j\in Hu^{i+1}$. This means that $\widehat{u^iv_j}=u^{i+1}$. Consequently
$$
B=\{ u^{n-1}, u^iv_ju^{-(i+1)}: 0\leq i\leq n-2,\ 1\leq j\leq s-1\}. 
$$
Now, recall that $G=Hu$. By \cite{Khod-Shah2} and its notations, the identity of $(G, \cdot)$is $u$ and we have $\overline{w}=w^{2-1n}$, for all $w\in G$. Also, we have $$
f(w_1, \ldots, w_n)=w_1\cdot \cdots \cdot w_n.
$$
So, using the relations of Sokolov, we have
\begin{eqnarray*}
w_1\circ w_2&=& f(w_1, \overline{u}, \stackrel{(n-3)}{u}, w_2=w_1u^{-1}w_2,\\
w^{-u}&=& uw^{-1}u,\\
\theta(w)&=&f(u,w, \overline{u}, \stackrel{(n-3)}{u})=uwu^{-1},\\
b&=&f(\stackrel{(n)}{u})=u^n.
\end{eqnarray*}
Again, if we use the notations of \cite{Khod-Shah2}, we have
$$
G=Hu=H_u, 
$$ 
and the map $\eta:Hto G$, defined by $\eta(w)=wu$, is an isomorphism. Hence, the set 
$$
B^{\prime}=\eta(B)=\{ u^n, u^iv_ju^{-i}, u^{n-2}v_ju: 0\leq i\leq n-3, 1\leq j\leq s-1\}
$$
is a basis of $(G, \circ)$. Using Nielsen transformations, the next set is also a basis:
$$
B^{\prime\prime}=\{ u^n, u^iv_ju^{-i}, (u^{n-2}v_ju)\circ (u^n)^{-u}:    0\leq i\leq n-3, 1\leq j\leq s-1\}.
$$
But, since
\begin{eqnarray*}
(u^{n-2}v_ju)\circ (u^n)^{-u}&=& u^{n-2}v_juu^{-1}u^{-n}u\\
                             &=&u^{n-2}v_ju^{-(n-2)},
\end{eqnarray*}
so we have
\begin{eqnarray*}
B^{\prime\prime}&=&\{ u^n, u^iv_ju^{-i}:    0\leq i\leq n-2, 1\leq j\leq s-1\}\\
                &=&\{ b, \theta^i(v)j): 0\leq i\leq n-2, 1\leq j\leq s-1\}.
\end{eqnarray*}
This proves the assertion 2. \\

Now, assume that 1 and 2 are true and we prove that 
$$
(G, f)=\mathrm{der}_{\theta, b}(G, \circ)
$$
is free. By 2, suppose $F=F(X)$ and $X=\{u, v_1, \ldots, v_{s-1}\}$, where $u$ is the identity of $(G, \circ)$. Consider the group $F_u$. Inside this group, we have
$$
w_1\ast w_2=w_1u^{-1}w_2.
$$
We know that $F\cong F_u$. So, $F_u$ is a free group. Define a homomorphism $\alpha:(G, \circ)\to (F_u, \ast)$ by 
$$
\alpha(b)=u^n, \ \alpha(\theta^i(v_j))=u^iv_ju^{-i}.
$$
Put $v_{ij}=\theta^i(v_j)$. For any $m$, we have
$$
\alpha(\theta^m(v_{ij}))=u^m\alpha(v_{ij})u^{-m}, 
$$
and hence for any $w\in G$, we have
$$
\alpha(\theta^m(w))=u^m\alpha(w)u^{-m}.
$$
This shows that for all $w_1, \ldots, w_n\in G$, 
\begin{eqnarray*}
\alpha(f(w_1, \ldots, w_n))&=&\alpha(w_1\circ \theta(w_2)\circ \cdots\circ \theta^{n-2}(w_{n-1})\circ b\circ w_n)\\
                           &=&\alpha(w_1)u^{-1}\alpha(\theta(w_2))u\cdots u\alpha(\theta^{n-2}(w_{n-1}))u^{-1}\alpha(b)u^{-1}\alpha(w_n)\\
                           &=&\alpha(w_1)u^{-1}(u\alpha(w_2)u^{-1})u\cdots \alpha(w_{n-1})u^{-(n-2)}u^{n-2}\alpha(w_n)\\
                           &=&\alpha(w_1)\cdots \alpha(w_n).
\end{eqnarray*}
Note that $\alpha(u)=u$ and $\alpha(v_i)=v_i$, and hence $F(X)=\langle \alpha(G)\rangle$. On the other side, $F(X)$ has a normal subgroup $H$, with $H\cong (G, \circ)$, and $[F;H]=m$ such that $m|n-1$. Since $[F_u;G]=n-1$, so $m=n-1$ and this shows that $F$ is the Post's cover of $(G, f)$. Now, since $X\subseteq G\subseteq F$, by \cite{Artam}, $(G,f)$ is free.

\end{proof}

\end{document}